\documentclass{amsart}
\usepackage{amsmath,amsfonts,latexsym,amscd,color,amssymb,mathrsfs}
\usepackage{graphicx}
\usepackage{float}
\usepackage{enumerate}
\newtheorem{theorem}{Theorem}

\newtheorem{proposition}[theorem]{Proposition}
\newtheorem{corollary}[theorem]{Corollary}

\newtheorem{lemma}[theorem]{Lemma}
\newtheorem{remark}[theorem]{Remark}
\newtheorem{example}[theorem]{Example}

\DeclareMathOperator{\ii}{{i}}
\DeclareMathOperator{\e}{{e}}
\newcommand{\Comp}{{\mathbb{C} }}
\newcommand{\Real}{{\mathbb{R} }}

\newcommand{\norm}[1]{\left\|#1\right\|}
\newcommand{\cblue}{\color{blue}}
\newcommand{\cblack}{\color{black}}

\begin{document}
\title[Global properties of eigenvalues of  perturbations of matrices.]{Global properties of eigenvalues of parametric rank one perturbations for unstructured and structured matrices. }

\author[A.C.M. Ran]{Andr\'e C.M. Ran}
\address{Afdeling Wiskunde, Faculteit
    der Exacte Wetenschappen, Vrije Universiteit Amsterdam, De Boelelaan
    1111, 1081 HV Amsterdam, The Netherlands
and Research Focus: Pure and Applied Analytics, North West University, South Africa. \\ ORCID: 0000-0001-9868-8605. }
   \email{a.c.m.ran@vu.nl}. 
\author[M. Wojtylak]{Micha\l{} Wojtylak}
\address{Instytut Matematyki, Wydzia\l{} Matematyki i Informatyki, Uniwersytet Jagiello\'nski, 
ul. \L ojasiewicza 6, 30-348 Krak\'ow, Poland.\\
ORCID:  0000-0001-8652-390X}
\email{michal.wojtylak@gmail.com}


\begin{abstract}
General properties of eigenvalues of $A+\tau uv^*$ as functions of $\tau\in\Comp$ or $\tau\in\Real$ or $\tau=\e^{\ii\theta}$ on the unit circle are considered. In particular, the problem of existence of global analytic formulas for eigenvalues is addressed. Furthermore, the limits of eigenvalues with $\tau\to\infty$ are discussed in detail. The following classes of matrices are considered: complex (without additional structure), real (without additional structure), complex $H$-selfadjoint and real $J$-Hamiltonian.  
\end{abstract}

\maketitle

\begin{center}
{\it  Dedicated to  Henk de Snoo on the occasion of his 75th birthday. \\ With admiration and respect.}
\end{center}

\section{Introduction}

The eigenvalues of matrices of the form $A+\tau uv^*$, viewed as a rank one parametric perturbation of the matrix $A$, have been discussed in a vast literature.
We mention the classical works of Lidskii \cite{Lidskii}, Vishik and Lyusternik \cite{VL}, as well as the more general treatment of eigenvalues of perturbations of the matrix in the books by Kato \cite{Kato} and Baumg\"artel \cite{Baumgartel}. Recently, Moro, Burke and Overton returned to the results of Lidskii in a more detailed analysis \cite{Moro-Burke-Overton}, while Karow obtained a detailed analysis of the situation for small values of the parameter \cite{Karow} in terms of structured pseudospectra. 
Obviously, parametric perturbations  appear in many different contexts. The works most closely related to the current one concern rank two perturbations by Kula, Wojtylak and Wysocza\'nski \cite{KWW},  matrix pencils  by De Ter\'an, Dopico and Moro \cite{DeDM08} and  Mehl, Mehrmann and Wojtylak \cite{MMW,MMW2} and matrix polynomials by   by De Ter\'an and Dopico \cite{DeD10}.

While the local behaviour of eigenvalues is fully understood, the global picture still has open ends, cf. e.g. the recent paper by C.K. Li and F. Zhang \cite{LiZhang}. The main problem here is that the  eigenvalues cannot be defined neither analytically nor  uniquely, even if we restrict the parameter $\tau$ to  real numbers. As is well-known the problem does not occur in the case of Hermitian matrices where an analytic function of $\tau$ with Hermitian values has eigenvalues and eigenvectors which can be arranged such that they are analytic as functions of $\tau$ (Rellich's theorem) \cite{Rellich}. 
Other cases where the difficulty is detoured appear, e.g., in a paper by Gingold and Hsieh \cite{GiHs}, where it is assumed that all eigenvalues are real, or in the series of papers of de Snoo (with different coauthors) \cite{DHS1,DHS3,SWW,SWW2} where only one distinguished eigenvalue (the so called \emph{eigenvalue of nonpositive type}) is studied for all real values of $\tau$.

Let us review now our current contribution.
To understand the global properties with respect to the complex parameter $\tau$ we will consider parametric perturbations of two kinds: $A+t uv^*$, where $t\in\Real$, or $A+\e^{\ii \theta} uv^*$, where $\theta\in[0,2\pi)$. The former case was investigated already in our earlier paper \cite{RW}, we review the basic notions in Section~\ref{s:prel}. However, we have not found the latter perturbations in the literature. We study them in Section~\ref{s:ang}, providing elementary results for further analysis.

Joining these two pictures together leads to new results on global behaviour of the eigenvalues in Section \ref{s:global}. Our main interest lies in generic behaviour of the eigenvalues, i.e., we address a question what  happens when a matrix $A$ (possibly very untypical and strange) is fixed and two  vectors $u,v$ are chosen numerically (we intentionally do not use the word `randomly' here).  
One of our main results (Theorem~\ref{th:global}) shows that the eigenvalues of $A+\tau uv^\top$ can be defined globally as analytic functions in this situation for real $\tau$. On the contrary, if one restricts only to real vectors $u,v$ this is no longer possible (Theorem~\ref{real-imposs}). 

In Section~\ref{sec:intermezzo} we study the second main problem of the paper: the limits of eigenvalues for large values of the parameter. Although similar results can be found in the literature we have decided to provide a full description, for all possible (not only generic) vectors $u,v$. This is motivated by our research in the following Section \ref{s:str}, where we apply these results to various classes of structured matrices. We also note there the classes for which a global analytic definition of eigenvalues in not possible (see Theorem~\ref{str-imposs}). 
In Section~\ref{s:non} we apply the general results to the class of matrices with nonnegative entries.

Although we focus on parametric rank one perturbations, we mention here that
the influence of a possibly non-parametric rank one perturbation on the invariant factors of a matrix has a rich history as well, see, e.g., the papers by Thompson \cite{Thompson} and M. Krupnik \cite{Krupnik}. Together with the 
 the works by H\"ormander and Melin \cite{HM}, Dopico and Moro \cite{DM}, Savchenko \cite{Sa1,Sa2} and Mehl, Mehrmann, Ran and Rodman \cite{MMRR1,MMRR2,MMRR3} they constitute a linear algebra basis for our research, developed  in our previous paper \cite{RW}. 
What we add to these methods  is some portion of complex analysis, by using the  function 
$
Q(\lambda)=v^\top(\lambda I_n- A)^{-1}u
$ 
and its holomorphic properties. This idea came to us through multiple contacts and collaborations with Henk de Snoo (cf. in particular the line of papers \cite{HSSW,SWW,SWW2}), for which we express our gratitude here. 

\section{Preliminaries}\label{s:prel}
If $X$ is a complex matrix (in particular, a vector) then by $\bar X$ we define 
the entrywise complex conjugate of $X$, further we set  $X^*=\bar X^\top$. 
We will deal with  rank one perturbations 
$$
B(\tau)=A+\tau uv^*,
$$
with $A\in\Comp^{n\times n}$, $u,v\in\Comp^n$.    
The parameter $\tau$ is a complex variable, we will often write it as $t\e^{\ii \theta}$ and fix either one of $t$ and $\theta$. 
We review now some necessary background and fix the notation. 

Let a matrix $A$ be given. We say that a property (of a triple $A,u,v$) holds for {\em generic vectors $u,v\in\Comp^n$} if there exists a finite set of nonzero complex polynomials of $2n$ variables,  which are zero on all $u,v$ not enjoying the property. Note that the polynomials might depend on the matrix $A$. In some places below a certain property will hold for generic $u,\bar v$. This  happens as in the current paper we consider the perturbations $uv^*$, while in \cite{RW} $uv^\top$ was used (even for complex vector $u,v$) . In any case, i.e., either $u,v$ generic or $u,\bar v$ generic, the closure of the set of `wrong' vectors has an empty interior.

 By $m_A(\lambda)$ we denote the minimal
polynomial of $A$. 
Define 
\begin{equation}\label{formulap}
p_{uv}(\lambda)= v^* m_A(\lambda)(\lambda I_n -A)^{-1} u 
\end{equation}
and observe that it is a polynomial, due to the formula for the inverse of a Jordan block (cf. \cite{RW}). Let $\lambda_1 , \ldots ,
\lambda_r$ be the (mutually different) eigenvalues of $A$, and corresponding to the eigenvalue
$\lambda_j$, let $n_{j,1} \geq n_{j,2} \geq \cdots \geq n_{j,\kappa_j}$ be
the sizes of the Jordan blocks of $A$.
We shall denote the degree of the polynomial $m_A(\lambda)$ by $l$, so
$$
l=\sum_{j=1}^r n_{j,1}.
$$
Then 
\begin{equation}\label{degpuv}
\deg p_{uv}(\lambda)\leq l-1 
\end{equation}
and equality holds for generic vectors $u,v\in\Comp^n$,  see \cite{RW}.

It can be also easily checked (see \cite{Sa1} or  \cite{RW}) that the characteristic polynomial of $B(\tau)$ satisfies
\begin{eqnarray}\det(\lambda I_n-B(\tau))&=&
\det(\lambda I_n-A)\cdot \left(1-\tau v^*(\lambda I_n-A)^{-1}u\right)\nonumber\\
&=&
\frac{\det(\lambda I_n - A)}{m_A(\lambda)}\left( m_A(\lambda)-\tau p_{uv}(\lambda)\right).\label{nextgreatpoly}
\end{eqnarray}
Therefore,  
the eigenvalues of $A+\tau uv^*$ which are not eigenvalues of $A$, are
roots of the polynomial
\begin{equation}\label{nicepoly}
p_{B(\tau)}(\lambda) = m_A(\lambda)- \tau p_{uv}(\lambda) .
\end{equation}
Note that some eigenvalues of $A$ may be roots of this polynomial as well. 
Saying this differently, we have the following inclusion of spectra of matrices
\begin{equation}\label{sigmapppp}
\sigma(B(\tau))\setminus\sigma(A)\subseteq p_{B(\tau)}^{-1}(0) \subseteq \sigma (B(\tau)),\quad \tau\in\Comp,
\end{equation}
but each of these inclusions may be strict.
Further, let us call an eigenvalue of $A$ {\em frozen $($by $u,v)$} if it is an eigenvalue of $B(\tau)$ for every complex $\tau$. Directly from \eqref{nextgreatpoly} we see that each frozen eigenvalue is either a zero of $\det(\lambda I_n-A)/m_A(\lambda)$, then we call it {\em structurally frozen}, or an eigenvalue of both $m_A(\lambda)$ and $p_{uv}(\lambda)$, and then we call it {\em accidentally frozen}.   
Note that, due to a rank argument, $\lambda_j$ is structurally frozen if and only if it has more than one Jordan block in the Jordan canonical form. Although being structurally frozen obviously does not depend on $u,v$, the Jordan form of $B(\tau)$ at these eigenvalues may vary for different $u,v$, which was a topic of many papers, see, e.g., \cite{HM,Sa1,RW}.

 In contrast, generically  $m_A(\lambda)$ and $p_{uv}(\lambda)$ do not have a common zero \cite{RW}, i.e., a slight change of $u,v$ leads to defrosting of $\lambda_j$ (which explains the name {\em accidentally}). 
In spite of this, we still need to tackle such eigenvalues in the course of the paper. The main technical problem is shown by the following, almost trivial, example.

\begin{example}\rm
Let $A=0\oplus A_1$, where $A_1\in\Comp^{(n-1)\times(n-1)}$ has a single eigenvalue at $\lambda_1\neq 0$ with a possibly nontrivial Jordan structure and let $u=v=e_1$. The eigenvalues of $B(\tau)$ are clearly $\tau$ and $\lambda_1$ and  the eigenvalue $\lambda_1$ is accidentally frozen. Observe that if we define $\lambda_0(\tau)=\tau$ then for $\tau=\lambda_1$ there is a sudden change in the Jordan structure of $B(\tau)$ at $\lambda_0(\tau)$.  
\end{example}

 To handle the evolution of eigenvalues of $B(\tau)$ without getting into the trouble indicated above we introduce the rational function 
\begin{equation}\label{Q2}
Q(\lambda):=v^* (\lambda I_n-A)^{-1} u= \frac{p_{uv}(\lambda)}{m_A(\lambda)}.
\end{equation}
It will play a central role in the analysis. Note that $Q(\lambda)$ is a rational function with poles in the set of eigenvalues of $A$, but not each eigenvalue is necessarily a pole of $Q(\lambda)$.  More precisely, if $\lambda_j$ ($j\in\{1,\dots r\}$) is an accidentally frozen eigenvalue of $A$ then  $Q(\lambda)$ does not have a pole of the same order as the multiplicity of $\lambda_i$ as a root of $m_A(\lambda)$, i.e, in the quotient $Q(\lambda)=\frac{p_{uv}(\lambda)}{m_A(\lambda)}$ there is pole-zero cancellation. 

\begin{proposition}\label{AvsQ}
Let $A\in\Comp^{n\times n}$, let $\tau_0\in\Comp$, let $u,v\in\Comp^n$ and assume that $\lambda_0\in\Comp$ is not an eigenvalue of $A$. Then $\lambda_0$ is an eigenvalue of $A+\tau_0 uv^*$ of algebraic multiplicity $\kappa\in\{1,2,\dots\}$ if and only if
 \begin{equation}\label{Qcond}
 Q(\lambda_0)=\frac{1}{\tau_0},\  Q'(\lambda_0)=0,\dots ,Q^{(\kappa-1)}(\lambda_0)=0,\  Q^{(\kappa)}(\lambda_0)\neq0.   
 \end{equation}
If this happens, then $\lambda_0$ has geometric multiplicity one, i.e., $A+\tau_0 uv^*$ has a Jordan chain of size $\kappa$ at $\lambda_0$. Finally,  $\lambda_0$ is not an eigenvalue of $A+\tau_1 uv^*$ for all $\tau_1\in\Comp\setminus\{\tau_0\}$.
\end{proposition}

\begin{remark}\label{AvsQR}\rm
If $\kappa=1$ condition \eqref{Qcond} should be read as $Q(\lambda_0)=1/\tau_0$, $Q'(\lambda_0)\neq 0$. In this case the implicit function theorem tells us that the eigenvalues can be defined analytically in the neigbourhood of $\tau_0,\lambda_0$.
If $\kappa>1$ then the analytic definition is not possible and the eigenvalues expand as Puiseux series,
that is, they behave locally as the solutions of $(\lambda-\lambda_0)^\kappa =\tau-\tau_0$, see, e.g.,
\cite{Baumgartel,Kato,Knopp}.
\end{remark}

\begin{remark}\label{ex:para1}\rm
One may be also tempted to define the eigenvalues via solving the equation $Q(\lambda)=1/\tau$ at $\lambda_0$ being an accidentally frozen eigenvalue of $A$ for which $Q(\lambda)$ does not have a pole at $\lambda_0$.  This would be, however, a dangerous procedure, as $\lambda_0$ might get  involved in a larger Jordan chain.  For example let
$$
B(\tau)=\begin{bmatrix} 1 & 1  \\ 0 & \tau \end{bmatrix}
$$
with an accidentally frozen eigenvalue 1 and $Q(\lambda)=1/\lambda$. Here for $\tau=1$ we get a Jordan block of size 2, but clearly the eigenvalues in a neighbourhood of $\lambda_0=1$ and $\tau_0=1$ do not behave as $1$ plus the square roots of $\tau- 1$. For this reason we will avoid the accidentally frozen eigenvalues.
\end{remark} 

\begin{remark}\rm
Note that in case $m_A$ and $p_{uv}$ have no common zeroes, i.e., there are no accidentally frozen eigenvalues, $Q^\prime(\lambda)$ can be expressed in terms of $m_A$ and $p_{uv}$ as follows
\begin{equation}\label{Qprime}
Q^\prime(\lambda)= \frac{p_{uv}^\prime(\lambda)m_A(\lambda)-p_{uv}(\lambda)m_A^\prime(\lambda)}{m_A(\lambda)^2},
\end{equation}
where cancellation of roots between numerator and denominator occurs in an eigenvalue of $A$ when corresponding to that eigenvalue there is a Jordan block of size bigger than one.
\end{remark}

\begin{proof}[Proof of Proposition \ref{AvsQ}]
For the proof of  
the first statement we start from the definition of $Q(\lambda)$. Note that $m_A(\lambda_0)$
is necessarily non zero, and so $p_{uv}(\lambda_0)$ is non-zero as well. If $\lambda_0$ is an eigenvalue of $B(\tau_0)$ which is not an eigenvalue of $A$, then, since $p_{B(\tau_0)}(\lambda_0)=0$, we have from \eqref{Q2} that $Q(\lambda_0)=\frac{1}{\tau_0}$, which proves the first equation in \eqref{Qcond}.

Furthermore,
from the definition of $Q(\lambda)$ we have $p_{uv}(\lambda)-Q(\lambda)m_A(\lambda)$ is identically zero. So, for any $\nu $ also the $\nu $-th derivative is zero. By the Leibniz rule this gives
$$
p_{uv}^{(\nu )}(\lambda)- \sum_{j=0}^\nu  \begin{pmatrix} \nu \\ j\end{pmatrix} Q^{(j)}(\lambda)m_A^{(\nu -j)}(\lambda) =0.
$$
We rewrite this slightly as follows:
\begin{equation}\label{Leibniz}
p_{uv}^{(\nu )}(\lambda) -Q(\lambda)m_A^{(\nu )}(\lambda) =\sum_{j=1}^\nu  \begin{pmatrix} \nu \\ j\end{pmatrix} Q^{(j)}(\lambda)m_A^{(\nu -j)}(\lambda) .
\end{equation}

Now, if $\lambda_0$ is an eigenvalue of algebraic multiplicity $\kappa$ of $B(\tau_0)$ and not an eigenvalue of $A$, then for $j=0, 1, \ldots , \kappa -1$ we have $m^{(j)}_A(\lambda_0)-\tau_0 p_{uv}^{(j)}(\lambda_0)=0$. Take $\nu =1$ in \eqref{Leibniz}, and set $\lambda=\lambda_0$:
$$
p_{uv}^\prime (\lambda_0)-\frac{1}{\tau_0}m_A(\lambda_0)=Q^\prime(\lambda_0)m_A(\lambda_0).
$$
Since $m_A(\lambda_0)\not=0$ it now follows that $Q^\prime(\lambda_0)\not=0$ when $\kappa=1$, while $Q^\prime (\lambda_0)=0$ when $\kappa >1$. Now proceed by induction. Suppose we have already shown that $Q^{(i)}(\lambda_0)=0$ for $i=1, \ldots , k <\kappa-1$.  Then set $\nu =k+1$ in \eqref{Leibniz} to obtain, using the induction hypothesis, that
$$
0=p_{uv}^{(k+1)}(\lambda_0) -Q(\lambda_0)m_A^{(k+1)}(\lambda_0)
= Q^{(k+1)}(\lambda_0)m_A(\lambda_0).
$$
Once again using the fact that $m_A(\lambda_0)\not=0$, we have that $Q^{(k+1)}(\lambda_0)=0$.
Finally, for $\nu =\kappa$ in \eqref{Leibniz}, and using what we have shown so far in this paragraph, we have
$$
0\not=p_{uv}^{(\kappa)}(\lambda_0) -Q(\lambda_0)m_A^{(\kappa)}(\lambda_0)=Q^{(\kappa)}(\lambda_0)m_A(\lambda_0),
$$
and so \eqref{Qcond} holds. 

Conversly, suppose \eqref{Qcond} holds. Then by the definition \eqref{Q2} of $Q$ we have $m_A(\lambda_0)-\tau_0p_{uv}(\lambda_0)=0$, so by \eqref{nicepoly} $\lambda_0$ is an eigenvalue of $B(\tau_0)$. Moreover, by \eqref{Leibniz} we have $m_A^{(j)} (\lambda_0)-\tau_0 p_{uv}^{(j)}(\lambda_0)=0$ for $j=1, \ldots, \kappa_1$, while $m_A^{(\kappa)} (\lambda_0)-\tau_0 p_{uv}^{(\kappa)}(\lambda_0)=\tau_0Q^{(\kappa)}(\lambda_0)m_a(\lambda_0)\not=0$. Hence, $\lambda_0$ is an eigenvalue of $B(\tau_0)$ of algebraic multiplicity $\kappa$,
completing the proof of the first statement.

For the proof of the second statement, note that
as $\lambda_0 I_n-A$ is invertible, any rank one perturbation of $\lambda_0 I_n-A$ can have only a one dimensional kernel.
Therefore, the Jordan structure of the perturbation at $\lambda_0$ is fixed. The last statement for $\tau_1=0$ follows from the assumption that $\lambda_0\notin\sigma(A)$ and for $\tau_1\notin\{0,\tau_0\}$   directly from \eqref{Qcond}.
\end{proof}

The statements in Proposition \ref{AvsQ} can also be seen by viewing $1-\tau Q(\lambda)=1-\tau v^*(\lambda I_n-A)^{-1} u$ as a \emph{realization} of the (scalar) rational function $1-\tau Q(\lambda)$. From that point of view the connection between poles of the function and eigenvalues of $A$, respectively, zeroes of the function and eigenvalues of $B(\tau)=A+\tau uv^*$ is well-known. For an in-depth analysis of this connection, even for matrix-valued rational matrix functions, see \cite{BGKR}, Chapter 8. We provided above an  elementary proof of the scalar case for the reader's convenience. 

 Note the following example, now more involved than the one in Remark~\ref{ex:para1}.
 
\begin{example}\label{ex:para}\rm
In this example we return to the consideration of accidentally frozen eigenvalues.
Let
$A=\begin{bmatrix} 1 & 1 & 0 \\ 0 & 1 & 0 \\ 0 & 0 & 2\end{bmatrix}, u=v=e_1$. 
Then we have:
\begin{align*}
m_A(\lambda)&= (\lambda-1)^2(\lambda -2)=\lambda^3-4\lambda^2+5\lambda -2,\\
Q(\lambda)&=\frac{1}{\lambda-1}, \quad Q^\prime (\lambda)=-\frac{1}{(\lambda-1)^2},\\
p_{uv}(\lambda)&=(\lambda-1)(\lambda -2)=\lambda^2-3\lambda+2.
\end{align*}
Also $B(\tau)=A+\tau uv^*=\begin{bmatrix} \tau+1 & 1 & 0 \\ 0 & 1 & 0 \\ 0 & 0 & 2\end{bmatrix}$, which has eigenvalues $1, 2$ and $\tau+1$. 
Note that both $1$ and $2$ are, by definition, accidentally frozen eigenvalues, although their character is a rather different.

Let us consider Proposition \ref{AvsQ} for this example.  
Note that $Q^\prime(\lambda)$ has no zeroes, which tells us that there are no double eigenvalues of $B(\tau)$ which are not eigenvalues of $A$. However, note that 
the zeros of $m_A(\lambda)$ and $p_{uv}(\lambda)$ are not disjoined. In particular,
$$
p_{uv}^\prime(\lambda)m_A(\lambda)-p_{uv}(\lambda)m_A^\prime (\lambda) =(\lambda-1)^2(\lambda-2)^2,
$$
which detects the double eigenvalue of $B(0)$ at $\lambda_1=1$ and a double semisimple  eigenvalue of $B(1)$ at $\lambda_2=2$, however, as can be seen from \eqref{Qprime} the roots of this polynomial are cancelled by the roots of $m_A^2(\lambda)$.
\end{example}

\section{Angular parameter}\label{s:ang}
In this section we will study the perturbations of the form
$$
A+t \e^{\ii \theta}uv^*,\quad \theta\in[0,2\pi),
$$
where $t>0$ is a parameter. More precisely, we will be interested in the evolution of the sets
$$
\sigma(A,u,v;t)=\bigcup_{0\leq \theta <2\pi} \sigma(A+t \e^{\ii \theta}uv^*)
$$
with the parameter $t>0$.

It should be noted that the sets $\sigma(A,u,v;t)$ are strongly related to the pseudospectral sets as introduced in e.g., \cite{Karow}, Definition 2.1. In fact they can be viewed as the boundaries of pseudospectral sets for the special case of rank one perturbations. The interest in \cite{Karow}, see in particular the beautiful result in Theorem 4.1 there, is in the small $t$ asymptotics of these sets. Our interest below is hence more in the intermediate values of $t$ and in the large $t$ asymptotics of these sets. 

By $z_1,\dots, z_{d}$ we denote the (mutually different)  zeroes of $Q'(\lambda)$, note that some of them might happen to be accidentally frozen eigenvalues, a slight modification of Example~\ref{ex:para} is left to the reader, see also Remark~\ref{rem:new} below.
We define $t_j$ as
$$
t_j =\frac{1}{|Q(z_j)|},\   \quad j=1, \ldots ,d.
$$

We group some properties of the sets $\sigma(A,u,v;t)$ in one theorem. Below by a {\em smooth closed curve} we mean a $\mathcal{C}^\infty$--diffeomorphic image of a circle.  
\begin{theorem}\label{prop:angular} Let $A\in\Comp^{n\times n}$ and let $u,v\in\Comp^n$ be two nonzero vectors, then the following holds.
\begin{itemize}
\item[(i)]\label{i1}
For $t>0$, $t\neq t_j$ $(j=1,\dots, d)$ the set $\sigma(A,u,v;t)$ consists of a union of smooth closed algebraic curves that do not intersect mutually.  
\item[(ii)]\label{i2} For $t= t_j$ $(j=1,\dots, d)$ the set $\sigma(A,u,v;t)$ is locally diffeomorphic with the interval, except the intersection points at those  $z_i$ for which  $t_j=1/|Q(z_i)|$ (possibly there are several such $z_i$'s).
\item[(iii)] For generic $u,v\in\Comp^n$ and for all $j=1,\dots d$ the point $z_j$ is a double eigenvalue of  $A+\tau uv^*$, for $\tau=1/Q(z_j)$. Two of the curves $\sigma(A,u,v,t)$ meet for $t=t_j$ at the point $z_j$. These curves are at the point $z_j$ not differentiable, they make a right angle corner, and meet at right angles as well.   
\item[(iv)]\label{union}
\begin{equation}\label{sumt} 
\sigma(A)\cup\bigcup_{t>0} \sigma(A,u,v;t) \cup Q^{-1}(0)=\Comp.
\end{equation}
\item[(v)]\label{i3} The function $t\to \sigma(A,u,v;t)$ is continuous in the Hausdorff metric for $t>0$.
\item[(vi)] $ \sigma(A,u,v;t)$ converges to $Q^{-1}(0)\cup\{\infty\}$ with $t\to\infty$.
\end{itemize}
\end{theorem}

\begin{proof}
Statements (i) and (ii) become clear if one observes that 
  $$
  \sigma(A,u,v;t)=\left\{ z\in\Comp: \frac1{|Q(z)|}=t\right\},\quad t>0,
  $$
i.e., it is a level sets of the modulus of a rational
function $1/Q(z)$. Since these level sets can also be written as the set of all points $z\in\mathbb{C}$ for which $|m_A(z)|^2=t^2|p_{uv}(z)|^2$ it is clear that for each $t$ they are algebraic curves. For $t\not= t_j$ ($j=1, \ldots , d$) the curves have no self-intersection and hence are smooth.

Let us now prove (iii). First note that for generic $u,v\in\Comp^n$ there are no 
accidentally frozen eigenvalues, as remarked in the end of Section~\ref{s:prel}.
Hence, every eigenvalue of $A+\tau uv^*$ of multiplicity $\kappa$, which is not an eigenvalue of $A$, is necessarily a zero of $Q(\lambda)-1/\tau$ of multiplicity $\kappa$, see Proposition~\ref{AvsQ}. However, by Theorem 5.1 of \cite{RW} for generic $u,v\in\Comp^n$ all eigenvalues of $A+\tau uv^*$ which are not eigenvalues of $A$ are of multiplicity at most two, and by Proposition \ref{AvsQ} the geometric multiplicity is one. Therefore the meeting points are at $z_j$ with $Q'(z_j)=0$, $Q''(z_j)\neq 0$. The behaviour of the eigenvalue curves concerning right angle corners follows from the local theory on the pertubation of an eigenvalue of geometric and algebaic multiplicity two for small values of $t-t_j$ (see e.g., the results of \cite{Lidskii}, but in particular, because of the connection with pseudospectral see \cite{Karow}).

To see (iv) let $\lambda_0\in\Comp$ be neither an eigenvalue of $A$ nor a zero of  $Q(\lambda)$. Then $Q(\lambda_0)=1/\tau_0$ for some $\tau_0\in\Comp$, hence $\lambda_0\in \sigma(A,u,v;|\tau_0|)$.
Statement (v) follows from Proposition 2.3 part (c) in \cite{Karow}. To see (vi) note that $1/|Q(\lambda)|$, as an absolute value of a holomorphic function,  does not have any local extreme points on $\Comp\setminus p_{uv}^{-1}(0)$ and it converges to infinity with $|\lambda|\to\infty$. 
\end{proof}

In Section \ref{sec:intermezzo} we will study in detail the rate of convergence in point (v) above. 

\begin{example}\label{ex1}\rm
 Consider the matrix
$$A=\begin{bmatrix}
    -2   &  0   &  0 &    0\\
     0   &  0   &  0 &    0\\
     0   &  0   &  4  &   1\\
     0   &  0   &  0  &   4
\end{bmatrix}$$ and the vectors
$$
u =\begin{bmatrix}
  -0.2 + 0.7\ii \\
   1.5 - 1.2\ii \\
   1.5 + 0.5\ii \\
   1.5 + 1.5\ii 
\end{bmatrix}
\mbox{\ and\ }  
v =\begin{bmatrix}
   0.5 + 0.3\ii\\
   1 - 0.8\ii\\
   0.8 + 0.9\ii\\
  -0.3 - 1.2\ii
\end{bmatrix}.
$$
In Figure \ref{f1} one may find the graph of the corresponding function $|1/Q(\lambda)|$, and a couple of curves $\sigma(A,u,v,t)$ at values of $t$ where double eigenvalues occur. Observe that these curves are often called level curves or contour plots of the function $|1/Q(\lambda)|$.
\begin{figure}
\begin{center}
\includegraphics[height=6cm]{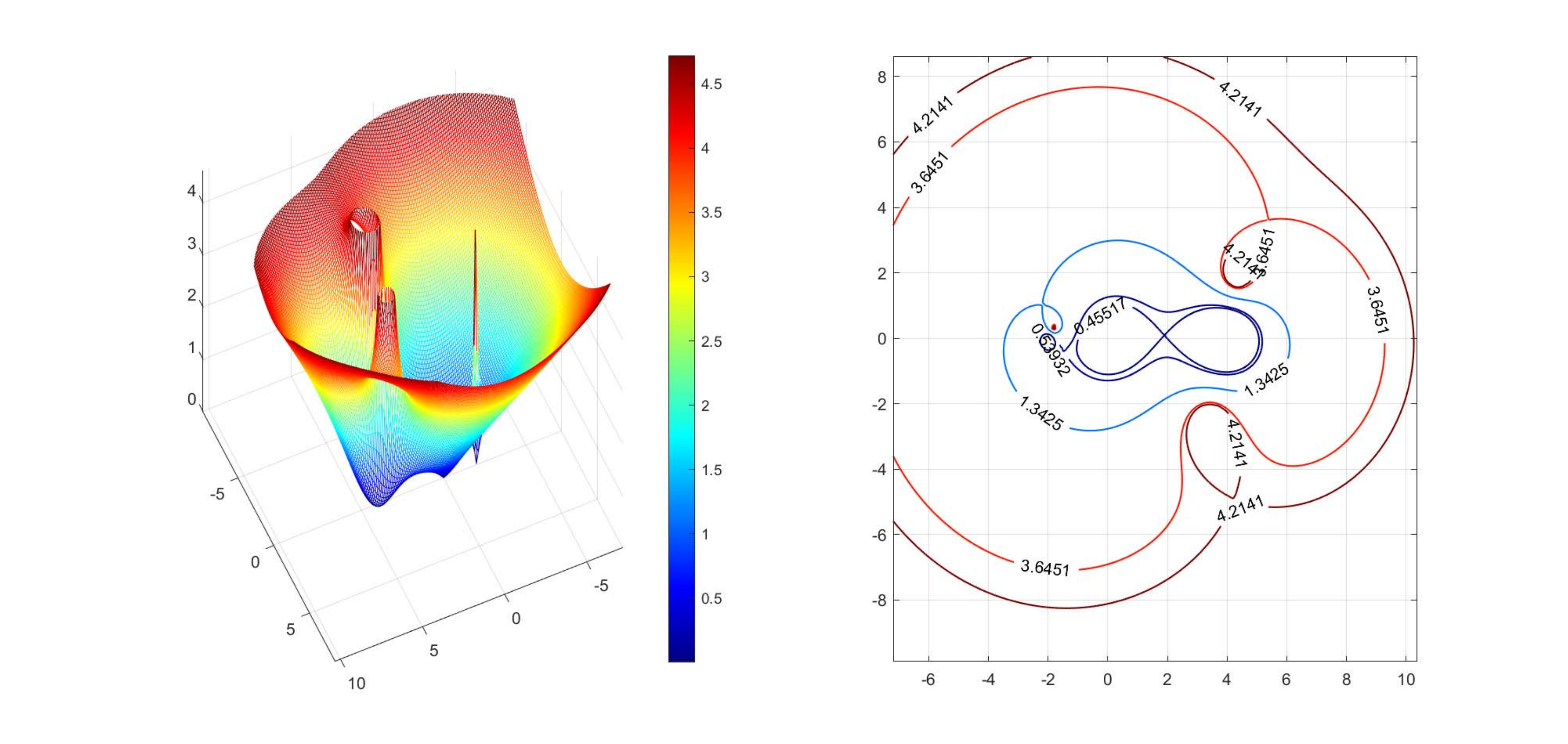}
\end{center}
\caption{The plot of $|1/Q(\lambda)|$ and the curves $\sigma(A,u,v,t)$ for the values of $t$ for which there is a double eigenvalue.}\label{f1}
\end{figure}
\end{example}

\begin{remark}\label{rem:new}\rm 
Observe that one may easily construct examples with $z_1,\dots,z_d$ given. 
 Let
$$
A=\begin{bmatrix} 0 & 1 & 0 & \cdots & 0 \\ \vdots& \ddots & \ddots & \ddots & \vdots \\ \vdots & & \ddots & \ddots & 0\\ 0 & \cdots & \cdots & 0 & 1\\
a_1 & a_2 & \cdots & \cdots & a_n
\end{bmatrix},
\quad
u=\begin{bmatrix} 0 \\ \vdots \\ \vdots \\ 0 \\ 1\end{bmatrix}, \quad
v=\begin{bmatrix} \overline{a}_1\\ \overline{a}_2\\ \vdots \\ \vdots \\ \overline{a}_n\end{bmatrix}
$$
with $a_1,\dots a_n\in\Comp\setminus\{0\}$.
Then for $\tau=1$ the matrix $B(\tau)=A+\tau uv^*$ is equal to the $n\times n$ Jordan block with eigenvalue zero, and hence has an eigenvalue of multiplicity $n$.  By a  construction similar to Example~\ref{ex:para} we may also make this eigenvalue accidentally frozen.
\end{remark}

\begin{example}\rm
In a concrete example, let 
$$
A=\begin{bmatrix} 0 & 1 & 0 \\ 0 & 0 & 1 \\ 1 & -1 & 1\end{bmatrix}, \quad u=\begin{bmatrix} 0 \\ 0 \\ 1 \end{bmatrix}, \quad v=\begin{bmatrix} 1 \\ -1 \\ 1\end{bmatrix}.
$$
Then $m_A(\lambda)=(\lambda -1)(\lambda^2+1)$ so the eigenvalues of $A$ are $1, \pm \ii$. Further $p_{uv}(\lambda)=\lambda^2-\lambda+1$ with roots at $\frac{1}{2}\pm \frac{\sqrt{3}}{2}\ii$, and the zeroes of $Q^\prime(\lambda)$ are the zeroes of $p_{uv}^\prime (\lambda)m_A(\lambda)-p_{uv}(\lambda)m_A^\prime(\lambda)=-\lambda^2(\lambda^2-2\lambda+3)$ which has roots at $0$ and at $1\pm\sqrt{2}\ii$. The corresponding values of $t$ are, respectively, $\frac{1}{|Q(0)|}=1$ and $\frac{1}{|Q(1\pm\sqrt{2}\ii)|}=\frac{4}{\sqrt{3}}$.
The eigenvalues of $A+\tau uv^*$ are plotted for the values $\tau=te^{\ii\theta}$ with $t=1$ and $t=\frac{4}{\sqrt{3}}$ in the graph below.
\begin{figure}
\includegraphics[height=6cm]{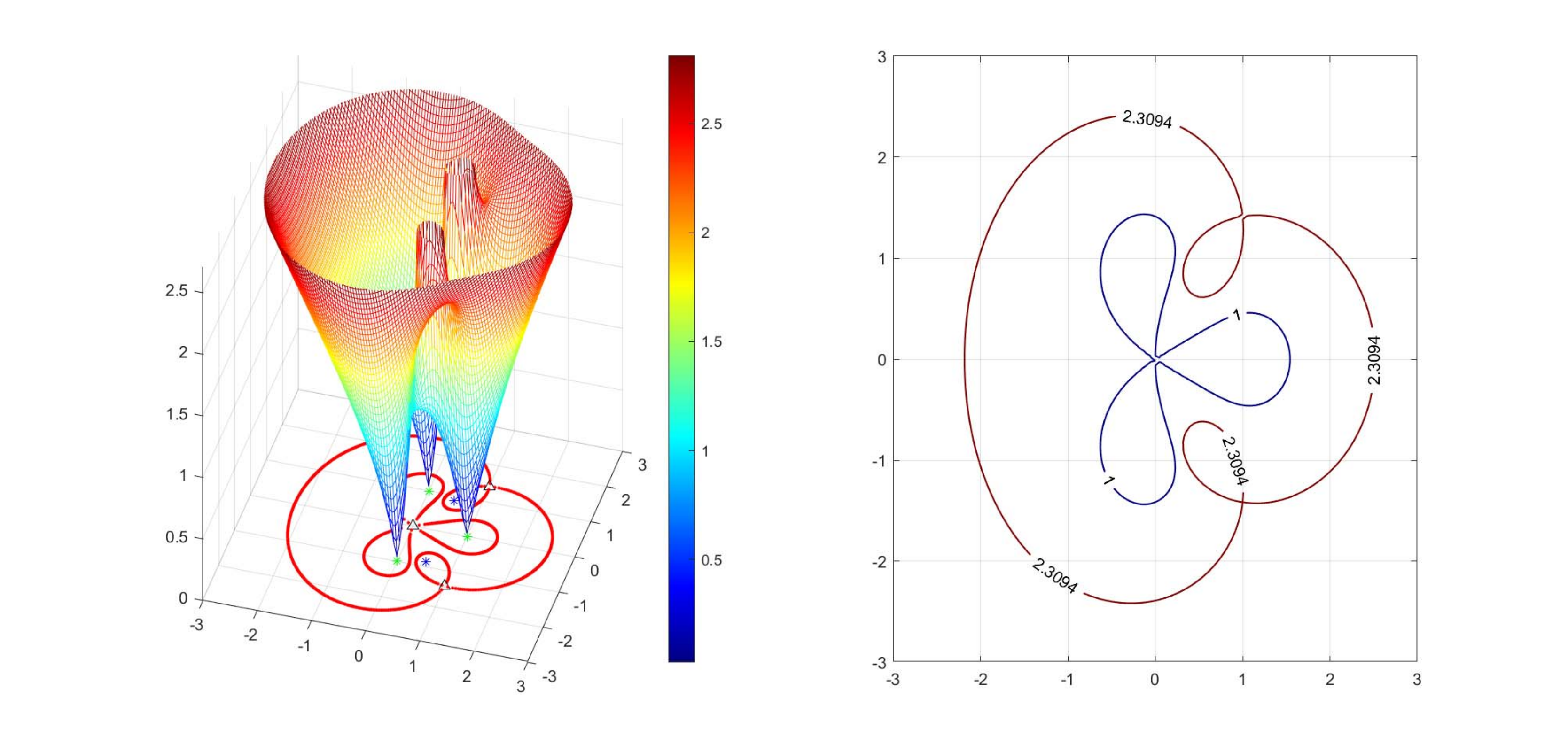}
\caption{Eigenvalue curves (right) showing a triple eigenvalue at zero for $\tau=1$ and double eigenvalues at $1\pm\sqrt{2}i$ for $\tau=\frac{4}{\sqrt{3}}$. 
On the left the graph of $1/|Q(\lambda)|$ with the same eigenvalue curves plotted in the ground plane. Green stars indicate the eigenvalues of $A$, blue stars the roots of $p_{uv}(\lambda)$ and triangles the zeroes of $Q^\prime(\lambda)$}\label{f2}
\end{figure}
\end{example}

\section{ Eigenvalues as global functions of the parameter}\label{s:global}

We return now to the problem of defining the eigenvalues as functions of the parameter. Recall that $l$ stands for the degree of the minimal polynomial of $A$. We start with the case where we consider the parameter $\tau$ to be real.

\begin{theorem}\label{th:global}
Let $A\in\Comp^{n\times n}$ and $u\in\Comp^n\setminus\{0\}$ be fixed.  Then for all $v\in\Comp^n$ except some closed set with empty interior the following holds.
\begin{itemize}
\item[(i)] The eigenvalues of
$$
B(\tau)=A+\tau uv^*,\quad \tau\in(0,+\infty),
$$
which are not eigenvalues of $A$, can be defined uniquely (up to ordering) as  functions $\lambda_1(\tau),\dots,\lambda_l(\tau)$ of the parameter $\tau\in(0,+\infty)$.
\item[(ii)] The remaining part of the spectrum of $B(\tau)$ consists of structurally frozen eigenvalues of $A$, i.e., there are no accidentally frozen eigenvalues (see formula \eqref{sigmapppp} and the paragraphs following it for definitions).
\item[(ii)] For $i,j=1,\dots d$, $i\neq j$ one has $\lambda_i(\tau)\neq \lambda_j(\tilde\tau)$ for all $\tau,\tilde\tau\in(0,+\infty)$. 
\item[(iii)] The functions $\lambda_j(t)$  can be extended to analytic functions in some open complex neighbourhood of $(0,+\infty)$.  
\end{itemize}
\end{theorem}

\begin{proof}
First let us write explicitly for which $u,v$ all the statements will hold.
Due to Proposition~\ref{AvsQ} and Remark~\ref{AvsQR} the necessary and sufficient condition for this is the following: there are no accidentally frozen eigenvalues and $Q(z_j)\notin\Real$ for all zeros $z_j$ of $Q'(\lambda)$. 
We will now show that given arbitrary $u_0,v_0$ which do not satisfy the above condition one may construct  $ u,v$, lying arbitrarily close to $u_0,v_0$ such that the condition holds on some open neighbourhood of $u,v$. We will do this in two steps.
First let us choose $\tilde u,\tilde v$ such that there are no accidentally frozen eigenvalues, i.e., there are no common eigenvalues of $m_A(\lambda)$ and $p_{\tilde u\tilde v}(\lambda)$. By \cite{RW} one may pick $\tilde u$ and $\tilde v$ arbitrarily close to $u_0,v_0$ and the desired property will hold in some small neighbourhood of $\tilde u,\tilde v$.
Furthermore,  $\tilde u,\e^{\ii\theta}\tilde v$ will also obey this property for all $\theta\in(-\pi,\pi)$.  
Note that one may find $\theta\neq 0$ arbitrarily small enough, so that with $v=\e^{\ii\tilde\theta}v$ ($|\theta-\tilde \theta|$ small enough) and $u=\tilde u$ one has
   $Q(z_i)\notin\Real$ for $i=1,\dots d$.   
\end{proof}

Observe that the statement  is essentially stronger and the proof is much  easier than in Theorem 6.2 of \cite{RW}. 

\begin{proposition}\label{prop:globalangular}
The statements of Theorem~\ref{th:global} are also true for the angular parameter, i.e., if one replaces $\tau\in(0,+\infty)$ by $\tau=\e^{\ii\theta}$, $\theta\in (0,2\pi)$ in all statements. 
\end{proposition}

\begin{proof}
The equivalent condition for all the statements is in this case: there are no accidentally frozen eigenvalues and $|Q(z_j)|\neq 1$ for all zeros $z_j$ of $Q'(\lambda)$. Hence, in the last step of the proof we need to replace $\tilde v$ by $t\tilde v$ with $t>0$ small enough.
\end{proof}

However, note that if we replace the complex numbers by the real numbers the statement  is false, as the following theorem shows.

\begin{theorem}\label{real-imposs}
Let $A\in\Real^{n\times n}$ and $u,v\in\Real^n$ be such that for some $\tau_0>0$ an analytic definition of eigenvalues of $A+\tau uv^\top$ is not possible due to
$$
Q(x)=1/\tau,\quad Q'(x)=0,\quad Q''(x)\neq 0
$$ for some $x\in\Real$ which is not an eigenvalue of $A$, cf. Remark \ref{AvsQR}.  Then for all $\tilde A\in\Real^{n\times n}$, $\tilde u\in\Real^n$, $\tilde v
\in\Real^n$ with $\norm{\tilde v-v}$, $\norm{\tilde u-u}$ and $\norm{\tilde A-A}$ sufficiently small the analytic definition of eigenvalues of $\tilde A+\tau \tilde u\tilde v^*$ is not possible due to existence of  $\tilde x\in\Real$, $\tilde \tau_0>0$, depending continuously on $\tilde A,\tilde u,\tilde v$ with 
$$
\widetilde Q(\tilde x)=1/{\tilde \tau_0},\quad \widetilde Q'(\tilde x)=0,\quad\widetilde  Q''(\tilde x)\neq 0,
$$
 where $\widetilde Q(z)$ corresponds to the perturbation $\tilde A+\tau\tilde u\tilde v^*$ as in \eqref{Q2}.
\end{theorem}

\begin{proof}
Recall the formulas \eqref{Q2} and \eqref{Qprime}
and set 
\begin{equation}\label{diff}
q_0(\lambda)=p_{uv}^\prime(\lambda)m_A(\lambda)-p_{uv}(\lambda)m_A^\prime(\lambda) ,
\end{equation}
so that
\begin{equation}\label{Q'}
Q^\prime(\lambda)= \frac{q_0(\lambda)}{m_A(\lambda)^2}.
\end{equation}
 By assumption we have that $m_A(x)\neq0$ and $q_0(x)=0$. We also get that $q_0'(x)\neq 0$, as otherwise $Q''(x)=0$.
We define $\tilde q_0(\lambda)$ analogously as $q_0(\lambda)$, i.e.,
$$
\widetilde Q^\prime(\lambda)=\frac{\tilde q_0(\lambda)}{m_{\tilde A}(\lambda)}.
$$
Both polynomials $\tilde q_0(\lambda)$ and $m_{\tilde A}(\lambda)$ have coefficients depending continuously on the entries of $\tilde A,\tilde u,\tilde v$.
As  $x$ is a simple zero  of the polynomial $q_0(\lambda)$, which is additionally real on the real line,  we have that there is a real $\tilde x$ near $x$ such that $\tilde q_0(\tilde x)=0$ and $m_{\tilde A}(\tilde x)\neq 0$ for $\tilde u$, $\tilde v$, $\tilde A$, as in the statement. Defining $\tilde t:=1/\widetilde Q(\tilde x)$ finishes the proof.
\end{proof}

\begin{remark}\rm
To give a punch line to Theorem~\ref{real-imposs} we make the obvious remark that $A,u,v$ satisfying the assumptions do exist. For each such $A$ the set of vectors $u,v\in\Real^n$ for which a double eigenvalue appears has a nonempty interior in $\Real^n$, contrary to the complex case discussed in Theorem~\ref{th:global}.  
\end{remark}

Note another reason for which the eigenvalues cannot be defined globally analytically  for real matrices.
\begin{proposition}\label{forcedcrossing}
Assume that the matrix $A\in\Real^{n\times n}$ has no real eigenvalues and let $u,v$ be two arbitrary nonzero real vectors. Then for some $\tau_0\in\Real\setminus\{ 0\}$ an analytic definition of eigenvalues of $A+\tau uv^\top$ is not possible due to 
$$
Q(x)=1/\tau,\quad Q'(x)=0
$$ for some $x\in\Real$, cf. Remark \ref{AvsQR}.  
\end{proposition}

\begin{proof}
Note that $Q(\lambda)$ is real and differentiable on the real line, due to the assumptions on $A$. As $Q(\tau)\to 0$ with $|\tau|\to\infty$, one has a local real extreme point of $Q(\lambda)$.
\end{proof}

\section{The eigenvalues of $A+\tau uv^*$ for large $|\tau|$}\label{sec:intermezzo}

We shall also be concerned with the asymptotic shape of the curves $\sigma(A,u,v;t)$.
The proof of the following result was given in \cite{vdCamp}:
let $A$ be an $n\times n$ complex matrix, let $u,v$ be generic complex $n$-vectors.
Asymptotically, as $t\to\infty$, these curves are circles, one with radius going to 
infinity centered at the origin, and the others with radius going to zero, and centers at the roots of $p_{uv}(\lambda)$.
The result will be restated in a more precise form below, in Theorem \ref{vAu}, part (v).
For this we first prove the following lemma.

\begin{lemma}\label{explicitformforpuv}
Let $m_A(\lambda)=\sum_{k=0}^l m_k\lambda^k$. Then 
\begin{align}\label{eq:puvexpl}
p_{uv}(\lambda)&= {\displaystyle \sum_{i=0}^{l-1} \left(\sum_{\substack {k-j=i+1\\ k,j\geq 0}} m_kv^*A^{j} u\right) }\lambda^i .
\end{align}
\end{lemma}

\begin{proof}
Recall that $p_{uv}(\lambda)=m_A(\lambda)v^*(\lambda  I_n-A)^{-1}u$.
Expanding $(\lambda  I_n-A)^{-1}$ in
Laurent series for $|\lambda|\geq \|{A}\|$ we obtain
$$
p_{uv}(\lambda)= m_A(\lambda)v^*(\lambda  I_n-A)^{-1}u
=\sum_{k=0}^l \sum_{j=0}^\infty
m_kv^*A^{j} u \lambda^{k-j-1}.
$$
Put $k-j-1=i$ and interchange the order of summation to see that 
$$
p_{uv}(\lambda)=
\sum_{i=-\infty}^{l-1} \left(\sum_{\substack{ k-j-1=i\\ k,j\geq 0}} 
m_k v^*A^{j}u\right)\lambda^i .
$$
However, $p_{uv}(\lambda)$ is a polynomial in $\lambda$, hence the  sum from $i=-\infty$ to
$-1$ vanishes, and we arrive at formula \eqref{eq:puvexpl}.
\end{proof}

Next, we analyze the roots of the polynomial $p_{B(\tau)}=
m_A(\lambda)- \tau p_{uv}(\lambda)$ as $\tau \to\infty$. We have already shown in 
\cite{RW} that if $u^*v\not= 0$, then $l-1$ of these roots will approximate
the roots of $p_{uv}(\lambda)$, while one goes to infinity. The condition $u^*v\neq 0$ obviously holds for generic $u,v$, however the next theorem presents the full picture in view of later applications to structured matrices. 

\begin{theorem}\label{vAu}
Let $A\in\Comp^{n\times n}$, $u,v\in\Comp^n$ and let $l\in\mathbb{N}$ denote the degree of the minimal polynomial $m_A(\lambda)$. Assume also that 
\begin{equation}\label{vAue}
v^*u=\cdots =v^*A^{\kappa-1}u=0,\quad v^*A^{\kappa}u\not= 0,
\end{equation}
for some $\kappa \in \{0,\dots,l-1 \}$
and put 
$$
v^*A^{\kappa}u=r_{\kappa}e^{\ii\theta_{\kappa}}.
$$ 
Then
\begin{itemize}
\item[(i)] $p_{uv}(\lambda)$ is of degree $l-\kappa-1$;
\item[(ii)] $l-\kappa-1$ eigenvalues  of $B(\tau)$ converge to the roots of $p_{uv}(\lambda)$ as
$\tau\to\infty$; 
\item[(iii)] there are $\kappa+1$ eigenvalues $\lambda_1(\tau),\dots,\lambda_{\kappa+1}(\tau)$ of $A+\tau uv^*$ which go to infinity with $r=|\tau|\to\infty$
as
$$
\lambda_j(re^{\ii\theta}) = \sqrt[\kappa+1]{rr_{\kappa}} e^{ \ii(\frac{1}{\kappa+1} (\theta+\theta_{\kappa}) +\frac{2j}{\kappa+1}\pi)}+O(1), \qquad j=1, 2, \ldots , \kappa+1,
$$
where $\theta\in[0,2\pi)$ is fixed,
and for all of them we have
$$
\frac{d\lambda_j}{d\tau}= \frac{v^*A^\kappa u}{l\lambda_j^{\kappa}} +O\left(\lambda^{-(k+1)}\right).
$$
so these eigenvalues can be parametrized by a curve
$$
\Gamma(\theta)= (rr_\kappa)^{\frac1{\kappa+1}}\exp(\ii\theta) + O(1),\quad  (r\to\infty);
$$
\item[(iv)]
as $\theta\to 2\pi$ one has, after possibly renumerating the eigenvalues \hfill\break $\lambda_1(\tau),\dots,\lambda_{\kappa+1}(\tau)$, that
$$
\lambda_j(re^{\ii\theta})\to \lambda_{j+1}(r), \quad j=1,\dots , \kappa, \quad \lambda_{\kappa +1} (re^{\ii\theta})\to \lambda_1(r);
$$
\item[(v)] additionally,  let $\zeta_1,\dots , \zeta_\nu$ denote the roots of the polynomial $p_{uv}(\lambda)$ with multiplicities respectively $k_1,\dots k_\nu$.
Denote
$$
v^*(\zeta_j I_n-A)^{-(k_j+1)}u=\rho_{j}e^{\ii \theta_j}, \qquad j=1, \ldots , \nu.
$$
Then 
$\sigma(A,u,v;t)$  for sufficiently large  $\tau$  can be parametrized by disjoint curves $\Gamma_1(\theta),\dots,\Gamma_{\nu+1}(\theta)$, where the $\kappa +1$ eigenvalues which go to infinity trace out a curve
$$
\Gamma_{\nu+1}(\theta)= (rr_\kappa)^{\frac1{\kappa+1}}\exp(\ii\theta) + O(1)
$$
while the $k_j$ eigenvalues near $\zeta_j$ trace out a curve $\Gamma_j(\theta)$ which is of the form
$$
\Gamma_j(\theta)=\zeta_j+|\tau|^{-\frac{1}{k_j}}\rho_j^{-\frac{1}{k_j}}e^{\ii \theta}+O\left(|\tau|^{-\frac{2}{k_j}}\right),\qquad 0\leq\theta\leq2\pi,
$$
 with $r=|\tau|\to\infty$.
\end{itemize}
\end{theorem}

\begin{proof}
Statement (i) results directly from Lemma \ref{explicitformforpuv}. 

Statement (ii) is a consequence of the fact that the characteristic polynomial of $A+\tau uv^*$ equals  $p_{B(\tau)}(\lambda) = m_A(\lambda)- \tau p_{uv}(\lambda)$ (formula \eqref{nicepoly}). Then, by \cite[Section II.1.7]{Kato} statement (ii) follows. 

For statement (iii) and following, by \eqref{eq:puvexpl} 
\begin{align*}
p_{B(\tau)}(\lambda)& = \sum_{i=0}^l m_i\lambda^1 -\tau\sum_{i=0}^{l-1} \left(
\sum_{\substack {k-j=i+1\\ k,j\geq 0}} m_kv^*A^ju\right) \lambda^i \\
&=\lambda^l +\sum_{i=0}^{l-1}\left( m_i-\tau \left( 
\sum_{\substack {k-j=i+1\\ k,j\geq 0}} m_kv^*A^ju\right) \lambda^i \right).
\end{align*}
In case $v^*u=0, v^*Au=0, \ldots , v^*A^{\kappa -1 } u=0$ and $v^*A^\kappa u\not=0$, this becomes 
\begin{align*}
p_{B(\tau)}(\lambda) = \lambda^l &+\lambda^{l-1}m_{l-1} + \cdots + \lambda^{l-\kappa}m_{l-\kappa} + \\
&+\lambda^{l-\kappa-1}(m_{l-\kappa-1}-\tau v^*A^\kappa u) + \mbox{lower order terms} .
\end{align*}
For large values of $\lambda$ and $\tau$ the dominant terms are $\lambda^{l-\kappa-1} (\lambda^{\kappa+1}-\tau v^*A^\kappa u)$,
showing that indeed, the largest roots behave asymptotically as the $(\kappa+1)$-th roots of $\tau v^*A^\kappa u$.

 Moreover, for the derivative of $\lambda$ with respect to $\tau$ we have by the implicit function theorem
$$
\frac{d\lambda}{d\tau} = -{\frac{\partial p_{B(\tau)}}{\partial \tau}}/{\frac{\partial p_{B(\tau)}}{\partial \lambda}}
= \frac{p_{uv}(\lambda)}{m_A^\prime(\lambda)-\tau p_{uv}^\prime(\lambda)}.
$$
Now $p_{uv}(\lambda)=\lambda^{l-\kappa-1}v^*A^\kappa u +\mbox{lower order terms}$, and 
$$
m_A^\prime(\lambda)-\tau p_{uv}^\prime(\lambda) =l\lambda^{l-1} +\mbox{lower order terms}.
$$
So
$$
\frac{d\lambda}{d\tau} =\frac{\lambda^{l-\kappa-1}v^*A^\kappa u +\mbox{lower order terms} }{l\lambda^{l-1} +\mbox{lower order terms} }.
$$
Dividing by $\lambda^{l-\kappa-1}$ in numerator and denominator we arrive at 
$$
\frac{d\lambda}{d\tau} =\frac{v^*A^\kappa u}{l} \frac{1}{\lambda^\kappa} + O\left(\lambda^{-(k+1)}\right).
$$
which concludes the proof of part (iii).

Note that also part (iv) follows directly from part (iii) combined with the continuity of the eigenvalues as a function of $\theta$ (even when taking $\theta\in\mathbb{R}$, rather than restricting to 
$\theta\in [0, 2\pi)$).

For part (v) and the remaining parts of part (iii) in the general case, recall that the eigenvalues of $A+\tau uv^*$ which are not eigenvalues of $A$ are (among the) roots of $m_A(\lambda)-\tau p_{uv}(\lambda)$, which are the same as the roots of the polynomial
$\frac{1}{\tau}m_A(\lambda)-p_{uv}(\lambda)$. First we take $\theta=0$, so $\tau=t$.

Put $s=1/t$, and consider the polynomial $sm_A(\lambda)-p_{uv}(\lambda)$ as a small perturbation of the polynomial $-p_{uv}(\lambda)$. By general theory concerning the behavior of the roots of a polynomial under such a perturbation (see, e.g., \cite{Knopp}, \cite{Baumgartel} Appendix, or Puiseux's original 1850 paper) for small $s$ the roots of $sm_A(\lambda)-p_{uv}(\lambda)$ near the roots of $p_{uv}(\lambda)$ are locally described by a Puiseux series of the form
$$
\zeta_j+c_{1j}s^{\frac{1}{k_j}}+c_{2j}s^{\frac{2}{k_j}}+\cdots , \qquad j=1, \ldots , \nu
$$
with $c_{1j}\not= 0$. Here $k_j$ is the multiplicity of $\zeta_j$ as a root of $p_{uv}(\lambda)$, and there are $k_j$ roots of $sm_A(\lambda)-p_{uv}(\lambda)$ near $\zeta_j$.

Next we do not consider $t \in \mathbb{R}$ but $\tau=te^{\ii\theta} \in \mathbb{C}$. Replacing $u$ by $e^{\ii\theta}u$
we see that the roots of $m_A(\lambda)-\tau p_{uv}(\lambda)$ for large $\tau$ near $\zeta_j$ behave as $\zeta_j+c_{1j}\tau^{-1/k_j}$. For fixed $|\tau|=r$ these $k_j$ roots near $\zeta_j$ trace out a curve $\Gamma_j(\theta)=\zeta_j+c_{1,j}
r^{-1/k_j}exp(\ii\theta)+o(r^{-1/k_j})$ with $r\to\infty$. 

We shall make these arguments much more precise as follows. Remember that an eigenvalue $\lambda_0$ of $A+\tau uv^*$ which is not also an eigenvalue of $A$ is a solution to $Q(\lambda)=1/\tau$. Consider large values of $|\tau|$ and consider also the large eigenvalues of $A+\tau uv^*$, for instance for $\tau$ large enough there is at least one eigenvalue with $|\lambda|>\|A\|$. Then
$\lambda$ satisfies
$$
1=\tau v^*\sum_{j=0}^\infty \frac{A^j}{\lambda^{j+1}}u=\tau\sum_{j=\kappa}^\infty  \frac{v^*A^j u}{\lambda^{j+1}}
$$
by the definition of $\kappa$. Hence
\begin{equation}\label{lambdatauinf}
\lambda^{\kappa+1}=\tau v^*A^\kappa u+\frac{\tau}{\lambda} v^*A^{\kappa+1}u+ \frac{\tau}{\lambda^2} v^*A^{\kappa+2}u+\cdots ,
\end{equation}
and so 
$$
\left( \frac{\lambda^{\kappa+1}}{\tau}-v^*A^\kappa u\right)=O\left(\frac{1}{\lambda}\right)
$$
Thus
$$
\lambda\approx (\tau)^{\frac{1}{\kappa+1}}(v^*A^\kappa u)^{\frac{1}{\kappa+1}}.
$$
Again, we can be much more precise than this: we know that $\lambda$ as a function of $\tau$ has a Puiseux series expansion, and if we set
$$
\lambda=c_{-1}\tau^{\frac{1}{\kappa+1}}+ c_0+ c_1\tau^{-\frac{1}{\kappa+1}}+\cdots
$$
one checks from the equation \eqref{lambdatauinf} that the following hold: 
\begin{align*}
c_{-1}&=(v^*A^ku)^{\frac{1}{\kappa+1}},\\
c_0&=\frac{1}{k+1}\cdot\frac{v^*A^{k+1}u}{v^*A^ku},\\
c_1&=\frac{1}{k+1}\cdot\frac{1}{(v^*A^ku)^{\frac{k+2}{k-1}}}\cdot (v^*A^{k+2}u-(k+2)v^*A^{k+1}u).
\end{align*}
Hence
$$
\lambda = (\tau)^{\frac{1}{\kappa+1}}(v^*A^\kappa u)^{\frac{1}{\kappa+1}} +
\frac{1}{k+1}\cdot\frac{v^*A^{k+1}u}{v^*A^ku}+O\left(\tau^{-\frac{1}{k+1}}\right)
$$
This completes the proof of part (iii), and gives the precise form of $\Gamma_{\nu+1}(\theta)$ for $r=|\tau|$ large enough.

Next we consider the eigenvalues of $A+\tau uv^*$ which are close to $\zeta_j$ for large $\tau$. Recall that $\zeta_j$ is a root of $p_{uv}(\lambda)$, so $m_A(\zeta_j)v^*(\zeta_j I_n-A)^{-1}u=0$.
If $\zeta_j$ would be a zero of $m_A(\lambda)$, then $\zeta_j$ is an accidentally frozen eigenvalue for $A,u,v$ and so is an eigenvalue of $B(\tau)$ for all $\tau$. Otherwise, $\zeta_j$ is not an eigenvalue of $A$, and we have $v^*(\zeta_j I_n-A)^{-1}u=0$. For $\lambda$ near $\zeta_j$ write 
\begin{align*}
v^*(\lambda I_n-A)^{-1}u & = v^*((\lambda-\zeta_j)+(\zeta_j I_n-A))^{-1}u\\ &=
v^*((\lambda-\zeta_j)(\zeta_j I_n-A)^{-1}+I)^{-1}(\zeta_j I_n-A)^{-1}u\\
&=v^*\sum_{k=0}^\infty (\lambda-\zeta_j)^k(\zeta_j I_n-A)^{-(k+1)}u\\
&=v^*(\zeta_j I_n-A)^{-1}u+v^*\sum_{k=1}^\infty (\lambda-\zeta_j)^k(\zeta_j I_n-A)^{-(k+1)}u.
\end{align*}
Since the first term is zero, we have
$$
v^*(\lambda I_n-A)^{-1}u =(\lambda-\zeta_j)v^*(\zeta_j I_n-A)^{-2)}u +(\lambda-\zeta_j)^2v^*(\zeta_j I_n-A)^{-3)}u+\cdots
$$
Again use the fact that any eigenvalue of $A+\tau uv^*$ which is not an eigenvalue of $A$ satisfies
$$
\frac{1}{\tau}=v^*(\lambda  I_n-A)^{-1}u.
$$
So, if the root $\zeta_j$ of $p_{uv}(\lambda)$ has multiplicity $k_j$, then 
$$
\frac{1}{\tau}=(\lambda-\zeta_j)^{k_j}v^*(\zeta_j I_n-A)^{-(k_j+1)}u +(\lambda-\zeta_j)^{k_j+1}v^*(\zeta_j I_n-A)^{-(k_j+2)}u+\cdots.
$$
For the moment, let us denote $v^*(\zeta_j I_n-A)^{-k}u$ by $a_{j,k}$. We know from the considerations in an earlier paragraph of the proof that $\lambda$ can be expressed as a Puiseux series in $\tau^{-1}$, let us say
$$
\lambda =\zeta_j+c_{1,j}\tau^{-\frac{1}{k_j}}+c_{2,j}\tau^{-\frac{2}{k_j}}+\cdots .
$$
Then $\lambda-\zeta_j=c_{1,j}\tau^{-\frac{1}{k_j}}+c_{2,j}\tau^{-\frac{1}{k_j}}+\cdots$, and inserting that in the above equation we obtain 
\begin{align*}
\frac{1}{\tau}&= \frac{1}{\tau}c_{1,j}^{k_j}a_{j,k_j+1} + k_jc_{1,j}^{k_j-1}\tau^{-\frac{k_j-1}{k_j}}\cdot c_{2,j}\tau^{-\frac{2}{k_j}}a_{j,k_j+1} +\\& \  + c_{1,j}^{k_j+1}\tau^{-\frac{k_j+1}{k_j}}a_{j,k_j+2} + {\rm smaller\ order\ terms}= \\
&=\frac{1}{\tau}c_{1,j}^{k_j}a_{j,k_j+1}
+\tau^{-\frac{k_j+1}{k_j}}\left( k_j c_{1,j}^{k_j-1} c_{2,j}a_{j,k_j+1} +  c_{1,j}^{k_j+1}a_{j,k_j+2}\right) +\cdots .
\end{align*}
Equating terms of equal powers in $\tau$ gives
$$
c_{1,j}=a_{j,k_j+1}^{-k_j}=\left(\frac{1}{v^*(\zeta_j I_n-A)^{-(k_j+1)}u}\right)^{\frac{1}{k_j}} 
$$
and using this we can derive a formula for $c_{2,j}$, which after some computation becomes
$$
c_{2,j}=-\frac{1}{k_j}\cdot\frac{v^*(\zeta_j I_n-A)^{-(k_j+2)}u}{v^*(\zeta_j I_n-A)^{-(k_j+1)}u}.
$$

Let us denote
$$
v^*(\zeta_j I_n-A)^{-(k_j+1)}u=\rho_{j}e^{\ii \theta_j}, \qquad j=1, \ldots , \nu.
$$
Then $c_{1,j}=\rho_j^{-\frac{1}{k_j}} e^{\ii \frac{\theta_j}{k_j}}$ and the $k_j$ eigenvalues near $\zeta_j$ trace out a curve $\Gamma_j(\theta)$ which is of the form
$$
\Gamma_j(\theta)=\zeta_j+|\tau|^{-\frac{1}{k_j}}\rho_j^{-\frac{1}{k_j}}e^{\ii \theta}+O\left(|\tau|^{-\frac{2}{k_j}}\right),\qquad 0\leq\theta\leq2\pi.
$$
This completes the proof of part (v).
\end{proof}

\begin{remark}\rm
As an alternative argument to much of the results of the previous theorem in the generic case,
consider $\tilde{B}(s)=uv^*+sA$, where $s=\frac{1}{\tau}=\frac{1}{t}e^{-\ii\theta}$.
For $t\rightarrow\pm\infty$ and thus $s\approx0$ 
we will denote 
the eigenvalues of $\tilde{B}(s)$ by $\nu_j(s)$ for $j=1,\ldots,n$. Note that there is a close relationship between $\nu_j(\tau^{-1})$ and $\lambda_j(\tau)$, namely, $\lambda_j(\tau)=\tau\nu_j(\tau^{-1})$.

Consider $\tilde{B}(s)$ as a perturbation of $uv^*$. Note that $uv^*$ has eigenvalues $vu^*$ and $0$, the latter with multiplicity $n-1$, and that generically, when $vu^*\not= 0$, $uv^*$ is diagonalizable. In the non-generic case, when $v^*u=0$, $uv^*$ has only eigenvalue $0$, with one Jordan block of size two, and $n-2$ Jordan blocks of size one. 

In the
generic case where $v^*u \neq 0$,
according to \cite{Kato}, Section II.1.2, in particular formula (1.7) there, also Section II.1.7 and Theorem 5.11 in Section II.5.6, and Lidskii's theorem \cite{Lidskii}, see also \cite{Baumgartel}, and for a nice exposition \cite{Moro-Burke-Overton},
we have that $B(s)$ has $n$ separate eigenvalues given by
\begin{equation*}
 \sigma(\tilde{B}(s)) = \{\nu_j(s)\} = \bigg\{ \begin{array}{cc}
\mbox{ }\mbox{ }\mbox{ } v^*u + sk_{1,1} + s^2k_{2,1} + \cdots \mbox{ } &  j=1 \mbox{ }\mbox{ }
\mbox{ }\mbox{ }\mbox{ }\mbox{ }\mbox{ }\mbox{ }\mbox{ } \\
\mbox{ }\mbox{ }\mbox{ }\mbox{ }\mbox{ }\mbox{ } 0 + sk_{1,j} + s^2k_{2,j} + \cdots \mbox{ } &  j=2,\ldots,n. \\
\end{array}
\end{equation*}
First we take $\theta=0$, so $\tau=t$.
In that case $s=\frac{1}{\tau}=\frac{1}{t}$. For $j=2,\ldots,n$ we have 
$\nu_j(\frac{1}{t}) = \frac{1}{t}k_{1,j} + \frac{1}{t^2}k_{2,j} + \cdots$ and therefore
\begin{equation*}
 \lambda_j(t) = t\nu_j(\frac{1}{t}) = k_{1,j} + \frac{1}{t}k_{2,j} + \cdots.
\end{equation*}
This works the same for $j=1$, then $\lambda_1(t)=tv^*u+k_{1,1}+\frac{1}{t}k_{2,1}+\cdots$.

Now consider the limit of $\lambda_j(t)$ as $t\rightarrow\pm\infty$ for $j=2, \ldots , n$.
By (ii) this is one of the roots of $p_{uv}(\lambda)$ or an eigenvalue of $A$. Generically, the roots of $p_{uv}(\lambda)$ will be simple.
After possibly rearranging the eigenvalues we may assume that for $j=2, \ldots , l$ the eigenvalue $\lambda_j(t)$ converges to one of the roots of $p_{uv}(\lambda)$, while for $j=l+1, \ldots , n$ the eigenvalue $\lambda_j(t)$ is constantly equal to an eigenvalue of $A$. Then 
\begin{equation*}
 \lim_{t\rightarrow\infty} \lambda_j(t) = \lim_{t\rightarrow\infty} t\nu_j(\frac{1}{t}) = k_{1,j} \mbox{ for } j=2,\ldots, l
\end{equation*}
where $k_{1,j}$ is either one of the roots of $p_{uv}(\lambda)$ or an eigenvalue of $A$. Thus $\lambda_j(t) = k_{1,j} + \frac{1}{t}k_{2,j}+O\left(\frac{1}{t^2}\right)$ for $j=2,\ldots, l$. 

Next we do not consider $t \in \mathbb{R}$ but $\tau=te^{\ii\theta} \in \mathbb{C}$.  Make the following transformation:
\begin{equation*}
 \sigma(A + \tau uv^*) = \sigma(A + te^{\ii\theta}uv^*) = \sigma(A + t\tilde{u}v^*)
\end{equation*}
where $\tilde{u} = e^{\ii\theta}u$. Note that $p_{uv}(\lambda)$ and $p_{\tilde{u}v}(\lambda)$ only differ by the constant $e^{\ii\theta}$ and so they have the same roots.
Applying the arguments from the previous paragraphs we obtain
\begin{equation*}
 \lambda_1(\tau) = \lambda_1(te^{\ii\theta}) = te^{\ii\theta}v^*u +O(1) 
\end{equation*}
and
\begin{equation*}
 \lambda_j(\tau) = \lambda_j(te^{\ii\theta}) = k_{1,j} + \frac{1}{t}e^{-\ii\theta}k_{2,j}+O\left(\frac{1}{t^2}\right), \mbox{ }\mbox{ } j=2,\ldots, l.
\end{equation*}
Consider for fixed $t$ the curve
\begin{equation*}
  \zeta_{2,t} = \{ \lambda_j(te^{\ii\theta}) \mid j=2,\ldots,n,\mbox{ } 0 \leq \theta < 2\pi \}.
\end{equation*}
The arguments above show that asymptotically $\zeta_{2,t}$ is a circle with radius $\frac{1}{t}k_{2,j}$ centered at $k_{1,j}$, which is a root of $p_{uv}(\lambda)$.
\end{remark}
\cblack

\begin{remark}\rm
Note that if \eqref{vAue} holds for $\kappa=l-1$ then $p_{uv}(\lambda)\equiv 0$ and by \eqref{nicepoly} the characteristic polynomial of the perturbed matrix coincides with the characteristic polynomial of $A$.
\end{remark}

\begin{remark}\rm
Consider the case $\kappa=2$. Then the eigenvalues that
go to infinity will trace out a circle, but each of them only traces out half a circle. In addition the speed with which the eigenvalues go to zero is considerably slower than
when $\kappa=1$.
\end{remark}

\begin{example}\rm
As an extreme example, consider $A=J_n(0)$, the $n\times n$ Jordan block with zero eigenvalue, and let $u=e_n$ and $v=e_1$, where $e_j$ is the $j$'th unit vector. Then 
$p_{uv}(\lambda)=1$, and the eigenvalues 
of $A+\tau uv^*$ are the $n$'th roots of
$\tau$, $\lambda_k(\tau)= \sqrt[n]{r}e^{i\frac{2k}{n}\pi}$ for $k=1, \cdots ,n$,
and $\frac{d\lambda}{d\tau} = \frac{1}{n}\tau^{\frac{1}{n}-1}= 
\frac{e_1^*A^{n-1}e_n}{n}\cdot \frac{1}{\lambda^{n-1}}$, as predicted by the theorem. 
\end{example}

\begin{example}\rm
Consider $A=I_2$, and the same $u$ and $v$ as in the previous example.
In this case $v^*A^ku=0$ for all $k$. Consequently, as is also immediate by
looking at the matrix, none of the eigenvalues moves. More generally, this happens
as soon as $A$ is upper triangular and $uv^*$ is strictly upper triangular.
\end{example}

\begin{example}\rm
As an example consider $A=\begin{bmatrix} 0 & 1 \\ -1 & -1\end{bmatrix}$, take $u=\begin{bmatrix} 0 \\1 \end{bmatrix}$ and $v=\begin{bmatrix} 1 \\ 1\end{bmatrix}$. Then $m_A(\lambda)=\lambda^2+\lambda+1$, $p_{uv}(\lambda)=\lambda+1$, and one computes $p_{uv}^\prime(\lambda)m_A(\lambda)-p_{uv}(\lambda)m_A^\prime(\lambda) =-\lambda(\lambda+2)$. So $Q^\prime (0)=0, Q^\prime(-2)=0$, and so $z_1=0,z_2=-2$. Finally the corresponding $t_1=\frac{1}{|Q(0)|}=1$, and $t_2=\frac{1}{|Q(-2)|}=3$.
The eigenvalue curves are shown below. Note the difference in scales between the individual graphs.
\begin{figure}
\begin{center}
\includegraphics[width=12cm,height=8cm]{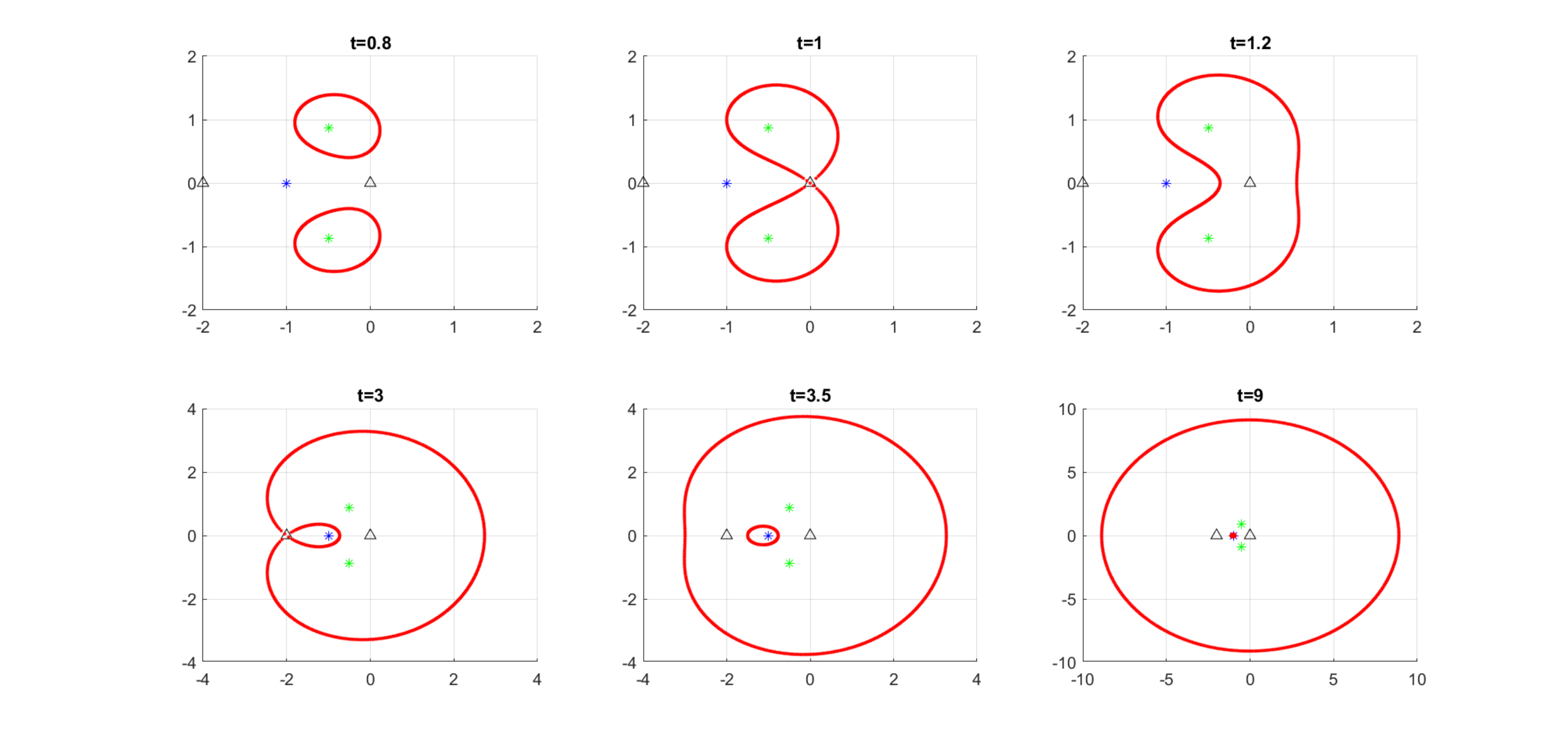}
\end{center}
\caption{Eigenvalue curves of $A+te^{\ii \theta}uv^*$ as functions of $\theta$ for several values of $t$. Green stars indicate the eigenvalues of $A$, the blue star indicates the root of $p_{uv}(\lambda)$ and the triangles the zeroes of $Q^\prime(\lambda)$. The final graph also illustrates the asymptotic circular behaviour for large values of $t$.}
\end{figure}
\end{example}

\section{Structured matrices}\label{s:str}
Generic rank one perturbations for several classes of structured matrices were
studied intensively over the last decade. We refer the reader to:  \cite{MMRR1} for  complex $J$-Hamiltonian complex $H$-symmetric matrices, \cite{MMRR2} for complex $H$-selfadjoint matrices
including the effects on the sign characteristic, \cite{MMRR4} for complex
$H$-orthogonal and complex $J$-symplectic  as well as complex $H$-unitary matrices, 
\cite{MMRR5} for the real cases including the effects on the sign characteristic, and
\cite{FGJR} for the case of $H$-positive real matrices. In \cite{BMRR} higher rank perturbations of structured matrices were considered. Another type of structure was
treated in \cite{BR}, where nonnegative rank one perturbations of $M$-matrices
are discussed.  Finally, the quaternionic case was discussed in \cite{MR}.

In the present section we will treat the classes of complex $H$-selfadjoint and real $J$-Hamiltonian matrices, analysing the global definition of eigenvalues and the convergence for large values of the parameter. First let us recall the definitions. 

We say that an $n\times n$ matrix $A$ is 
\begin{itemize}
\item[(H)] \emph{$H$-selfadjoint} if $A\in\Comp^{n\times n}$,  $HA=A^*H$, where $H\in\Comp^{n\times n}$ is some Hermitian nonsingular matrix.
\item[(J)]  \emph{$J$-Hamiltonian} if  $A\in\Real^{n\times n}$, $JA=-A^\top J$, where $J\in\Real^{n\times n}$ is a  nonsingular real matrix  satisfying $J=-J^\top$.
\end{itemize}
 Note that rank one matrices in these classses are, respectively, of the form 
 \begin{itemize}
\item[(H)] $uu^*H$, for some $u\in\Comp^n\setminus\{0\}$,
\item[(J)] $uu^\top J$ for some $u\in\Real^n\setminus\{0\}$.
\end{itemize}
 Consequently,  the function $Q(\lambda)$ takes, respectively, the form
 \begin{itemize}
\item[(H)] $Q(\lambda)=u^*H(\lambda I_n-A)^{-1}u$,
\item[(J)] $Q(\lambda)=u^\top J(\lambda I_n-A)^{-1}u$.
\end{itemize}

It appears that in both these classes global analytic definition of eigenvalues is not a generic property, similarly to  Theorem~\ref{real-imposs} in the real unstructured case.  By inspection one sees that the proof remains almost the same, the key issue is that all polynomials involved are real on the real line and $x$ is a simple real zero of $q_0(\lambda)$.

\begin{theorem}\label{str-imposs} Assume one of the following
\begin{itemize}
\item[(H)]  $A\in\Comp^{n\times n}$ is $H$-selfadjoint with respect to some nonsingular Hermitian $H$, and $u\in\Comp^n\setminus\{0\}$,
\item[(J)] $A\in\Real^{n\times n}$ is $J$-Hamiltonian with respect to some nonsingular skew-symmetric $J$, and $u\in\Real^n\setminus\{0\}$. 
\end{itemize}
If for some $\tau_0>0$ an analytic definition of the eigenvalues  of $A+\tau_0 uv^*$ is not possible due to 
$$
Q(x)=1/\tau_0,\quad Q'(x)=0,\quad Q''(x)\neq 0
$$ 
for some $x\in\Real$, cf. Remark \ref{AvsQR},  then for all 
\begin{itemize}
\item[(H)] $\tilde A\in\Comp^{n\times n}$ being  $H$-selfadjoint, $\tilde u\in\Comp^n$
\item[(J)] $\tilde A\in\Real^{n\times n}$ being $J$-symmetric, $\tilde u\in\Real^n$ $($respectively$)$
\end{itemize} 
with  $\norm{\tilde u-u}$ and $\norm{\tilde A-A}$ sufficiently small the analytic definition of the eigenvalues is not possible due to  existence of $\tilde x\in\Real$, $\tilde \tau_0>0$, depending continuously on $\tilde A,\tilde u,\tilde v$ with 
$$
\widetilde Q(\tilde x)=1/{\tilde \tau_0},\quad \widetilde Q'(\tilde x)=0,\quad\widetilde  Q''(\tilde x)\neq 0,
$$
 where $\widetilde Q(z)$ corresponds to the perturbation of $\tilde A$ as described above the theorem.
\end{theorem}

\begin{remark}\rm We remark here, that Proposition~\ref{forcedcrossing}
holds for $J$-Hamiltonian matrices as well, and also
 for complex  $H$-selfadjoint matrices. 
\end{remark}
\cblack

We continue the section with corollaries from Theorem \ref{vAu}. While statement (ii) below is not surprising if one takes into account the symmetry of the spectrum of a $J$-Hamiltonian matrix with respect to both axes, statement (i) cannot be derived using symmetry principles only. 
\begin{corollary} 
\begin{itemize}
\item[(i)] Let $A\in\Real^{n\times n}$, consider the perturbation $A+\tau uu^\top J$, where $J$ is real, nonsingular and skew symmetric, and $u\in\Real^n\setminus\{0\}$ and $\tau\in\mathbb{R}$. Then  there are $($at least$)$ two eigenvalues of $A+t uu^\top J$ going to infinity as described by part {\rm (ii)} of Theorem~\ref{vAu}.
\item[(ii)]  If, additionally to {\rm (i)}, $A$ is also $J$-Hamiltonian 
the number of such eigenvalues is even, and 
\item[(iii)] In case $A$ is $J$-Hamiltonian and $u^\top JAu >0$ then there are two real eigenvalues converging to infinity as $\tau$ goes to $+\infty$, and two purely imaginary eigenvalues going to infinity as $\tau$ goes to $-\infty$. In case $u^\top JAu <0$ the situation is reversed.
\item[(iv)] In case $A$ is $J$-Hamiltonian and
$u^\top JAu=0$ there are at least four eigenvalues going to infinity as $\tau$ goes to $+\infty$ and as $\tau$ goes to $-\infty$. More precisely, let $\kappa$ be the first (necessarily odd) integer for which $u^\top JA^\kappa u\not= 0$. If $u^\top JA^\kappa u >0$ then for $\tau\to + \infty$ there are at least two real eigenvalues going to infinity, and two purely imaginary eigenvalues going to infinity. If $u^\top JA^\kappa u <0$ then for $\tau\to - \infty$ there are at least two real eigenvalues going to infinity, and two purely imaginary eigenvalues going to infinity. 
\end{itemize}
\end{corollary}

\begin{proof}
Part (i) follows from Theorem \ref{vAu} and the fact that for any vector $u$ we have $u^\top Ju=0$ by the skew-symmetry of $J$. 

Part (ii) follows from the same reasoning taking into account that for any even $k$ the matrix $JA^k$ is skew-symmetric, from which one has $u^\top JA^ku=0$ for even $k$.  

Parts (iii) and (iv) follow from Theorem \ref{vAu}, part (iii), using the fact that $Q(\lambda)$ is real on both real and imaginary axis. 
\end{proof}

\begin{example}\rm
Consider the following matrices $J$ and $A$ and vector $u$:
$$
J=\begin{bmatrix} 0 & 0 & 0 & 1 \\ 0 & 0 & -1 & 0 \\ 0 & 1 & 0 & 0 \\ -1 & 0 & 0 & 0\end{bmatrix},
\quad A=\begin{bmatrix} 1 & 1 & 0 & 0 \\ 0 & 1 & 0 & 0 \\ 0 & 0 & -1 & 1 \\ 0 & 0 & 0 & -1\end{bmatrix}, \quad u=\begin{bmatrix} 0 \\ 1 \\ 1 \\ 1 \end{bmatrix}.
$$
Then, one checks easily that $A$ is $J$-Hamiltonian, and that $u^\top JAu=0$, while $u^\top JA^3\cblue u\cblack=-4\not=0$. The polynomial $p_{uv}(\lambda)$ for $v=-Ju$ is constant, equal to $-4$. Hence all four eigenvalues of $A+t uu^\top J$ are going to infinity, as is shown in the following figure. Note also that the rate of convergence to infinity in this example should be as the fourth root of $t$, which is confirmed by the graph (the fourth root of $125000$ is about $19$).
\begin{figure}
\includegraphics[height=4cm]{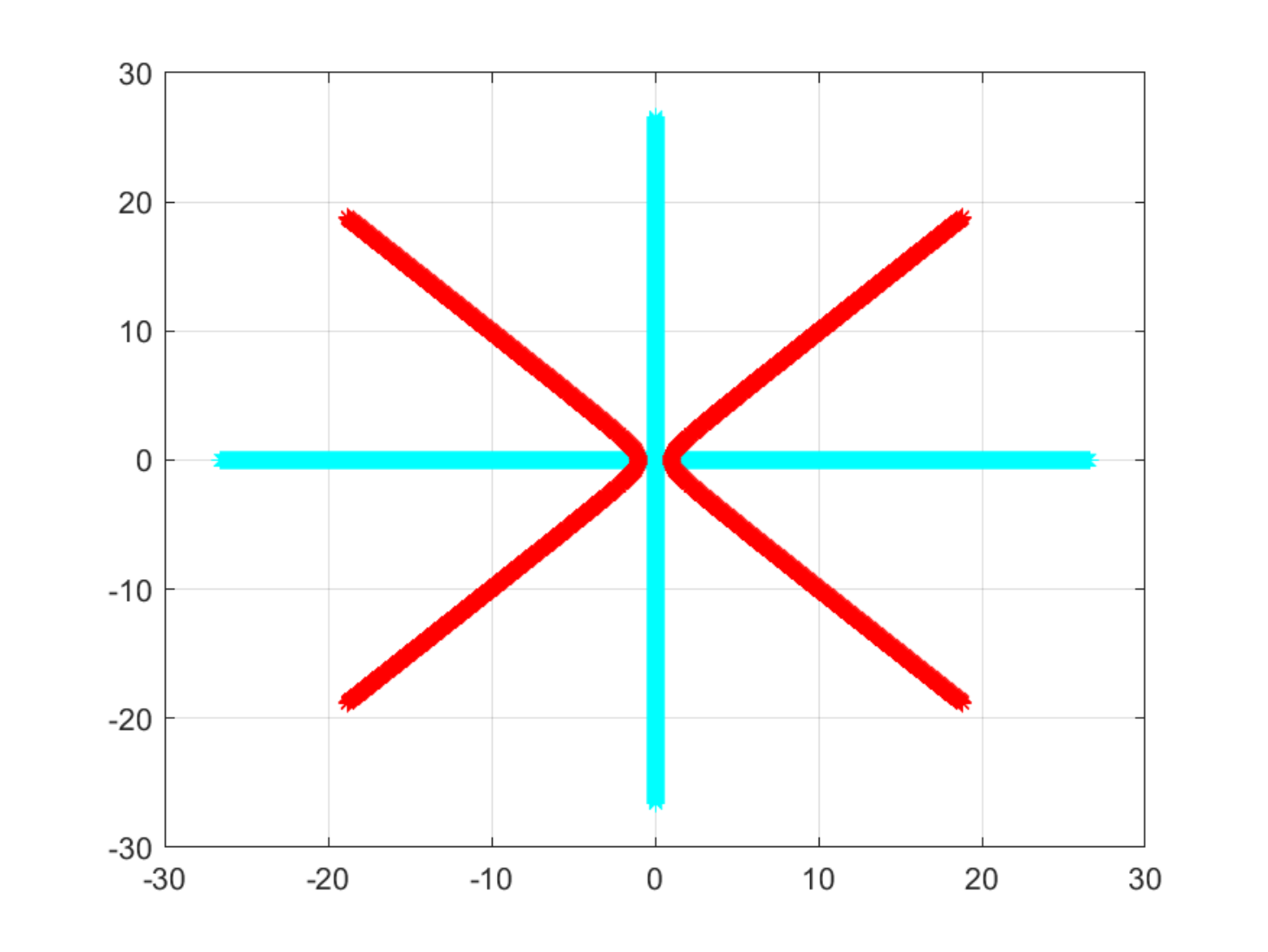}
\caption{Four eigenvalues going to infinity. The plot shows the eigenvalues of $A+tuu^\top J$ for $0\leq t\leq 125000$ in red, and the eigenvalues of $A-tuu^\top J$ for the same range of $t$ in cyan.}
\end{figure}
\end{example}

%
%
%

\section{Nonnegative matrices}\label{s:non}

We will apply Theorem \ref{vAu} to the 
setting
of nonnegative matrices.
Recall that a nonnegative matrix $A$ is called \textit{irreducible} if there is no permutation matrix $P$ such that $P^\top A P$ is of a block form 
$$
P^\top A P = \begin{bmatrix} X & Y \\ 0 & Z \end{bmatrix}
$$
with $X$ and $Z$ being nontrivial square matrices. By the  \textit{graph associated with the matrix }$A=[a_{ij}]_{ij=1}^n$ we understand the directed graph with vertices $1,\dots ,n$ and with the set of edges consisting of only those pairs $(i,j)$ for which $a_{ij}>0$. By a \textit{cycle} we understand a directed path from the vertex $i$ to itself.

\begin{theorem}
Let $A=[a_{ij}]_{ij=1}^n\in\Real^{n\times n}$ be a nonnegative, irreducible matrix. Let also $l$ denote the length of the shortest cycle in the graph of the matrix $A+e_{i_0}e_{j_0}^\top$ containing the $(i_0,j_0)$ edge. Then the matrix 
$$
A+\tau e_{i_0}e_{j_0}^\top,\quad \tau>0
$$
has precisely $l$ eigenvalues converging to infinity with $\tau\to+\infty$

\end{theorem}

\begin{proof}
Note that 
$$
e_{j_0}^\top A^k e_{i_0} = (A^k)_{j_0i_0}=0
$$
if  and only if  $k<l$, as $l$ is the length of the smallest cycle going through the $(i_0,j_0)$ edge.
By Theorem \ref{vAu} the matrix $A+\tau uv^\top$ has precisely $l$ eigenvalues converging to infinity.
\end{proof}

 Note that the number $l$ of eigenvalues converging to infinity may be  greater than the number of eigenvalues of $A$ on the spectral circle, i.e, the imprimitivity index. However, by the theory of nonnegative matrices $l$, as the length of the (shortest) cycle, is always a multiple of  the imprimitivity index, see, e.g., Theorem 1.6. of \cite{DJZ}.


\begin{example}\rm
Consider the matrix $A=\begin{bmatrix} 1 & 1 \\ 1 & 1 \end{bmatrix}$ and
the vectors $u=\begin{bmatrix} 1 \\ 0 \end{bmatrix}$ and $v=
\begin{bmatrix} 0 \\ 1 \end{bmatrix}$. Then $B(\tau) = \begin{bmatrix}
1 & 1+\tau \\ 1 & 1 \end{bmatrix}$. Then $v^\top u=0$, while $v^\top Au\not= 0$.
So both eigenvalues of $B(\tau)$ will go to infinity. For $\tau \geq 0$ the matrix
$B(\tau)$ is an entrywise positive matrix, so one of the eigenvalues will be
the spectral radius. By Theorem \ref{vAu}, both eigenvalues go to zero at the same rate,
but as the eigenvalues are $\sqrt{1+\tau} \pm 1$ their moduli are not equal.
\end{example}

\section*{ Acknowledgement} The authors have made use of the work of Brian van de Camp, who did his master thesis under their supervision. In particular, the result of Theorem \ref{vAu}, part (v) 
 was first presented in the thesis \cite{vdCamp} for the generic case only.
 
 \smallskip
 
\paragraph{\bf Declarations:} $ $\smallskip

\paragraph{\bf Conflict of interest} The authors declare that they have no conflict of interest.\smallskip

\paragraph{\bf Funding}
Not applicable. \smallskip

\paragraph{\bf Availability of data and material} Data sharing is not applicable to this article as no datasets were generated or analysed during the current study.\smallskip

\paragraph{\bf Code availability} Not applicable.

\end{document}